\documentclass[a4paper, 12pt]{article}

\usepackage[sort&compress]{natbib}
\bibpunct{(}{)}{;}{a}{}{,} 

\usepackage{amsthm, amsmath, amssymb, mathrsfs, multirow, url, subfigure}
\usepackage{graphicx} 
\usepackage{ifthen} 
\usepackage{amsfonts}
\usepackage[usenames]{color}
\usepackage{fullpage}

\theoremstyle{plain} 
\newtheorem{thm}{Theorem}

\theoremstyle{definition}

\theoremstyle{remark} 

\newtheorem*{astep}{A-step}
\newtheorem*{pstep}{P-step}
\newtheorem*{cstep}{C-step}

\newcommand{\prob}{\mathsf{P}}

\newcommand{\bel}{\mathsf{bel}}
\newcommand{\pl}{\mathsf{pl}}
\newcommand{\mpl}{\mathsf{mpl}}
\newcommand{\mbel}{\mathsf{mbel}}

\newcommand{\bin}{{\sf Bin}}
\newcommand{\unif}{{\sf Unif}}
\newcommand{\nm}{{\sf N}}

\newcommand{\chisq}{{\sf ChiSq}}

\newcommand{\RR}{\mathbb{R}}
 
\newcommand{\U}{\mathscr{U}}

\newcommand{\XX}{\mathbb{X}}

\newcommand{\UU}{\mathbb{U}}

\newcommand{\xbar}{\bar{x}}

\newcommand{\G}{\mathscr{G}}
\newcommand{\Gbar}{\overline{\mathscr{G}}}
\newcommand{\Gtilde}{\widetilde{\mathscr{G}}}

\renewcommand{\S}{\mathcal{S}}
\renewcommand{\SS}{\mathbb{S}}

\title{A mathematical characterization of confidence \\ as valid belief}
\author{Ryan Martin \\
Department of Statistics \\
North Carolina State University \\
\url{rgmarti3@ncsu.edu} 
}
\date{\today}

\begin{document}

\maketitle 

\begin{abstract}    
Confidence is a fundamental concept in statistics, but there is a tendency to misinterpret it as probability.  In this paper, I argue that an intuitively and mathematically more appropriate interpretation of confidence is through belief/plausibility functions, in particular, those that satisfy a certain validity property. Given their close connection with confidence, it is natural to ask how a valid belief/plausibility function can be constructed directly.  The inferential model (IM) framework provides such a construction, and here I prove a complete-class theorem stating that, for every nominal confidence region, there exists a valid IM whose plausibility regions are contained by the given confidence region.  This characterization has implications for statistics understanding and communication, and highlights the importance of belief functions and the IM framework.  

\smallskip

\end{abstract}

\section{Introduction}
\label{S:intro}  

Confidence is a fundamental concept in statistics, ``arguably the most substantive ingredient in modern model-based theory'' \citep{fraser2011.rejoinder}, dating back to \citet{neyman1941} and also, indirectly, to \citet{fisher1973}, through its close ties to fiducial inference \citep{zabell1992, seidenfeld1992, efron1998}.  However, like the controversial p-value \citep[e.g.,][]{ionides.response, pvalue.ban, wasserstein.lazar.asa}, interpretation of this fundamental concept is somewhat elusive.  For example, as instructors teaching confidence intervals to students in an introductory statistics course, we are careful to distinguish confidence from probability: ``95\% confidence'' {\em does not} mean that the unknown parameter resides in the stated interval with probability 0.95.  Unfortunately, there apparently is no fully satisfactory explanation of what the ``95\% confidence'' feature of the stated interval actually {\em does} mean.  In practice, a stated confidence interval is informally interpreted as a set of parameter values that, together, is ``sufficiently and justifiably believable'' or, equivalently, as a collection of parameter values that, individually, are ``sufficiently and justifiably plausible.''  Statisticians are reluctant to adopt the use of words like ``believable'' and ``plausible'' because of their seemingly non-scientific connotations, but remaining silent about the interpretation of confidence and, in particular, leaving the door open for incorrect interpretations, is not any more scientific.  Fortunately, belief, like probability, is a well-defined mathematical object, so if a rigorous connection can be made between it and confidence, then we have for our students a clear and honest explanation of what confidence means, consistent with how confidence intervals are used in practice.  Beyond the classroom, there is obvious value---both within the statistics community and in our efforts to communicate with others---in having an agreeable understanding of confidence that is both intuitive and mathematically precise.  A goal of this paper is to formally develop this connection between confidence and belief.  


To set the scene, suppose observable data $X \in \XX$ is modeled by a distribution $\prob_{X|\theta}$ indexed by a parameter $\theta \in \Theta$; here either $X$, $\theta$, or both can be scalars, vectors, or something else.  Let $\phi = \phi(\theta)$ be an interest parameter, taking values in $\Phi = \phi(\Theta)$, and assume existence of a family of set-valued maps $C_\alpha: \XX \to 2^\Phi$ where, for any $\alpha \in (0,1)$, $C_\alpha(X)$ is a $100(1-\alpha)$\% confidence region for $\phi$.  That is, 
\begin{equation}
\label{eq:coverage}
\inf_{\theta \in \Theta} \prob_{X|\theta} \{C_\alpha(X) \ni \phi(\theta)\} \geq 1-\alpha, \quad \forall \; \alpha \in (0,1). 
\end{equation}
Despite the warnings given in textbooks, many might be tempted to convert this ``$X|\theta$''-probability statement into a ``$\theta|(X=x)$''-probability statement by constructing a probability density function for $\phi=\phi(\theta)$, on $\Phi$, that has $C_\alpha(x)$ as its level-$(1-\alpha)$ contour, $\alpha \in (0,1)$.  This defines a so-called {\em confidence distribution} \citep[e.g.,][]{schweder.hjort.book} for $\phi$, but, for reasons described in Section~\ref{SS:not.prob}, it is risky to treat this as a genuine probability distribution.  Basically, certain operations afforded to probabilities are inconsistent with the features of confidence, so the issues with probability go beyond semantics.  Therefore, I conclude that it is inappropriate to convert confidence to probability.  \citet[][p.~74]{fisher1973} seems to agree with this conclusion when he writes: 

\begin{quote}
{\em [Confidence regions] were I think developed and advocated under the impression that in a wider class of cases they could provide information similar to that of the probability statements derived by the fiducial argument.  It is clear, however, that no exact probability statements can be based on them...}
\end{quote}

As an alternative to ordinary probability, here I will focus specifically on uncertainties described by distributions of random sets \citep{molchanov2005, nguyen.book, shafer1979}, a special case of the Dempster--Shafer theory of belief and plausibility functions \citep[e.g.,][]{shafer1976, dempster2008, dempster.copss}.  In particular, I will argue in Section~\ref{SS:then.what} that a belief function is more appropriate than a probability distribution for describing the uncertainty about the unknown parameter encoded in a confidence statement, both intuitively and mathematically \citep[cf.,][]{balch2012}.  In particular, confidence can be propagated through the operations afforded to belief functions, but not through the usual integration with respect to probability measures; see Section~\ref{SS:then.what}.

In addition to re-expressing confidence regions as belief functions, it is of interest to see what insights this connection provides.  Certainly not all belief functions will admit plausibility regions, defined in \eqref{eq:plaus.region} and \eqref{eq:mpl.region}, that meet the coverage probability condition \eqref{eq:coverage}, so important questions include: what additional properties are needed?~and how can a belief function satisfying these properties be constructed?  Answers to both questions can be found in the recent work on {\em inferential models} \citep[IMs,][]{imbook}, and I review the relevant details in Section~\ref{S:source}.  Given that the IM approach provides a framework for constructing valid belief functions with plausibility regions satisfying \eqref{eq:coverage}, a relevant question is if there are confidence regions that cannot be obtained via the IM approach.  Analogous to the classical decision-theoretical results, I prove here, in Section~\ref{S:characterization}, the following complete-class theorem for IMs: given a confidence region $C_\alpha$ for $\phi=\phi(\theta)$, be it Bayes, fiducial, or whatever, as long as it satisfies \eqref{eq:coverage} and a mild compatibility condition, there exists a valid IM on the full parameter space $\Theta$ such that the corresponding (naive) marginal plausibility region for $\phi$ is contained by $C_\alpha$.  That is, at least for the purpose of uncertainty quantification via confidence regions, there is no loss of generality or efficiency in adopting an IM approach.  Therefore, based on this theorem and the fact that IMs provide more than just plausibility regions, I would argue, in the same spirit as \citet{fraser2011}, that {\em confidence regions are quick and dirty IMs}.  

While the focus here is on confidence and its ties to belief functions and IMs, there is a larger context in which the points here are relevant.  A primary source of disagreement between the different schools of thought---Bayes, frequentist, fiducial, etc---is the {\em source} of the probability used to describe uncertainty for the purpose of making inference, and how should those probabilities be interpreted.  But having different theories, and methods built from each that can give different results in applications, hurts our field's reputation in other areas of science \citep{fraser.copss}.  According to the results presented below, if confidence regions are an inferential objective---and they are in virtually every application of statistical inference---then probability is not the appropriate description of that type of uncertainty, so old questions like ``where should probability come from?''~and ``how should these probabilities be interpreted?''~are actually irrelevant.  This realization can potentially put an end to the debates over whose probability is ``right,'' and provide opportunities for better understanding and communication.

\section{What is confidence?}
\label{S:confidence}

\subsection{Not probability}
\label{SS:not.prob}

Introductory statistics textbooks make a distinction between confidence and probability, but what, then, are {\em confidence distributions}?  As the name suggests, a confidence distribution is a probability measure on the parameter space, derived from an $\alpha$-indexed family of confidence regions.  In the simplest context, given a set of upper confidence limits, a confidence distribution can be defined by taking the $100(1-\alpha)$th percentile of the distribution to be the corresponding $100(1-\alpha)$\% upper confidence limit, $\alpha \in (0,1)$.  More precise details about the construction of confidence distributions and their properties can be found in \citet{xie.singh.2012}, \citet{schweder.hjort.2002, schweder.hjort.book}, and \citet{nadarajah.etal.2015}.  Ultimately, this approach boils down to interpreting confidence regions in the way that students in introductory statistics courses are specifically warned not to, i.e., ``the parameter of interest falls in the 95\% confidence region with probability 0.95.''  The aforementioned warning is given to students because the stated inference problem does not come equipped with a probability measure on the parameter space, so probability statements about the parameter have no connection to the ``real world.''  The confidence distribution approach side-steps this issue by treating confidence as a user-defined subjective probability.  While there is nothing principally wrong with a subjective interpretation, there are various reasons to be concerned about calling this a ``probability,'' which I describe below.

First, there are cases where the only confidence regions that have the desired coverage probabilities are unbounded, so stacking these up does not lead to a genuine probability distribution---instead, one gets something with positive mass at ``infinity.''  One example of this phenomenon is the Fieller--Creasy problem \citep{fieller1954, creasy1954} which, in its simplest form, concerns inference about the ratio $\phi = \theta_1/\theta_2$ based on two independent normal observations, $X_1 \sim \nm(\theta_1,1)$ and $X_2 \sim \nm(\theta_2, 1)$.  In particular, for certain $\theta=(\theta_1,\theta_2)$, any $100(1-\alpha)$\% confidence interval for $\phi$ that has coverage probability $1-\alpha$ must be unbounded with positive probability under $\prob_{X|\theta}$.  This implies that certain non-extreme quantiles of the confidence distribution would be $\infty$, i.e., the corresponding confidence density for $\phi$ integrates to something less than 1, hence, it is not a genuine ``distribution.''  This phenomenon arises in the entire class of problems investigated in \citet{gleser.hwang.1987}, which includes regression models with measurement errors and other problems of interest in econometrics \citep{dufour1997}.  

Second, it is well known that confidence distributions for a full parameter cannot generally be marginalized like probabilities, via integration, to obtain a confidence distribution for an interest parameter $\phi=\phi(\theta)$.  A standard example is given in \citet{stein1959}, but here I will reconsider the  Fieller--Creasy problem described above.  It is straightforward to derive a joint confidence distribution for $\theta=(\theta_1,\theta_2)$ based on the two independent observations, $X=(X_1, X_2)$.  In particular, this confidence distribution would be $\nm_2(X, I_2)$, an independent bivariate normal with mean $X$ and unit variances.  If one derives a marginal distribution for $\phi = \theta_1/\theta_2$ from this joint distribution for $\theta$, then one gets the following distribution function for $\phi$, depending on $x=(x_1,x_2)$, 
\[ G_x(\varphi) = \int F(\varphi \, z - x_1) f(z - x_2) \,dz, \quad \varphi \in \RR, \]
where, $F$ and $f=F'$ denote the standard normal distribution and density functions, respectively; see, also, \citet{hinkley1969}.  However, 
\[ C_\alpha(x) = [G_x^{-1}(\tfrac{\alpha}{2}), G_x^{-1}(1-\tfrac{\alpha}{2})] \]
is {\em not} a $100(1-\alpha)$\% confidence interval for $\phi$ since \eqref{eq:coverage} is not satisfied.  The problem is that there are certain (extreme) values of $(\theta_1,\theta_2)$ such that the coverage probability is arbitrarily small; for example, if $\theta_1=1$ and $\theta_2=20$, then the interval above, with $\alpha=0.05$, has coverage probability approximately 0.12, so clearly \eqref{eq:coverage} fails.  

This known marginalization failure, 
along with the reality that users will inevitably be tempted to do this marginalization, has led some researchers to reject the notion of a joint confidence distribution, focusing instead on scalar confidence distributions.  But these marginalization issues can also arise in scalar problems.  For example, there is a reasonable confidence distribution for $\theta$ based on $X \sim \nm(\theta,1)$, but integration does not lead to a genuine confidence distribution for $\phi=|\theta|$.  Indeed, Figure~\ref{fig:abstheta} shows the distribution of the confidence cumulative distribution function, denoted by ${\sf CD}(|\theta|)$, evaluated at the true $|\theta|$, under sampling from $\nm(0.5, 1)$.  Clearly this distribution is not uniform, hence the confidence distribution features are not preserved under marginalization, even in the scalar case. 

\begin{figure}[t]
\begin{center}
\scalebox{0.70}{\includegraphics{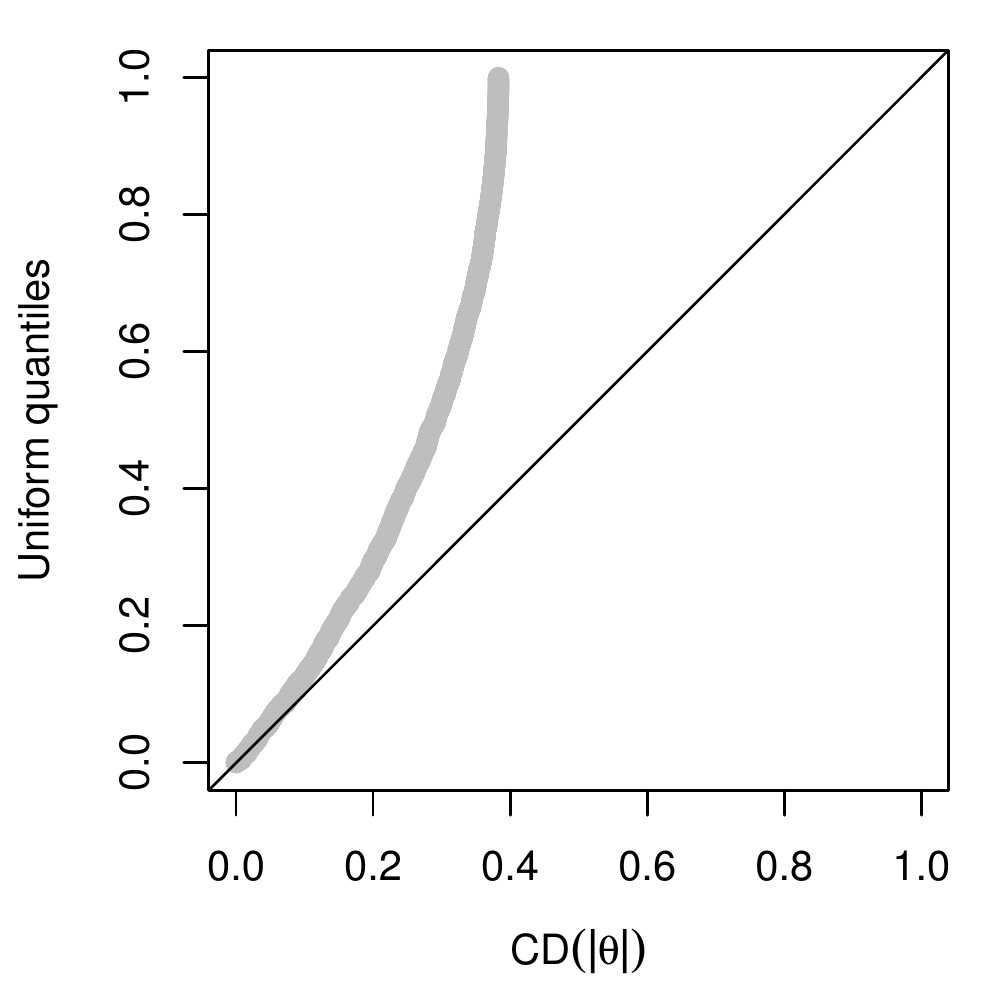}}
\end{center}
\caption{According to Definition~1 in \citet{xie.singh.2012}, ${\sf CD}(|\theta|)$ is a confidence distribution if the gray dots---its empirical distribution based on 5000 Monte Carlo samples under the scenario described in the text---lie on the diagonal line.}
\label{fig:abstheta}
\end{figure}


To summarize, while there are some simple cases where it might seem safe to think of a confidence distribution as a probability, the above points highlight that this probabilistic interpretation might be dubious.  If there are cases where a confidence distribution is, mathematically, not a distribution and, furthermore, even in cases where it is a distribution, it cannot be manipulated like one, then calling it a distribution is both inaccurate and potentially misleading.  And the same issues with marginalization hold for other inferential methods that summarize uncertainty via probability, including Bayes, fiducial, generalized fiducial \citep{hannig.review}, and maybe others; see \citet{balch.martin.ferson.2017}.  Fortunately, there is an alternative to probability that seems particularly suited for describing this type of inferential uncertainty.

\subsection{If not probability, then what?}
\label{SS:then.what}

For the present discussion, consider the case where interest is in the full parameter $\theta$.  Recall the standard connection between confidence regions and significance tests.  In particular, for a test that rejects the null hypothesis $H_0: \theta=\vartheta$ in favor of the alternative $H_1: \theta \neq \vartheta$, at level $\alpha \in (0,1)$, if $C_\alpha(x) \not\ni \vartheta$, the p-value is given by 
\begin{equation}
\label{eq:cr.contour}
p_x(\vartheta) = \sup\{ \alpha \in (0,1): C_\alpha(x) \ni \vartheta \}.
\end{equation}
Allowing $\vartheta$ to vary determines a ``p-value function'' \citep{pvalue.course} that has many other names, including {\em confidence curve} \citep{birnbaum1961, schweder.hjort.2002, schweder.hjort.2013, schweder.hjort.book}, {\em possibility function} \citep{zadeh1978}, {\em preference function} \citep{spjotvoll1983, blaker.spjotvoll.2000}, and {\em significance function} \citep{fraser1990, fraser1991}, among others.  Here, I opt for the {\em plausibility contour} terminology because there is a formal mathematical theory available for describing these objects that will provide some statistical insights.  Specifically, assume that $\{C_\alpha: \alpha \in (0,1)\}$ is nested in the sense that (i)~$C_\alpha(x) \subseteq C_{\alpha'}(x)$ for $\alpha \geq \alpha'$ and for all $x$, and (ii)~there exists $\hat\theta(x)$ in $C_\alpha(x)$ for all $\alpha \in (0,1)$.  For example, Bayesian highest posterior density credible regions have this property, with $\hat\theta(x)$ being the maximum {\em a posteriori}, or MAP, estimator.  Note that $\hat\theta(x)$ need not be unique and, in the case where $C_\alpha(x)$ is ``one-sided,'' like a confidence lower/upper bound, then $\hat\theta(x)$ can be infinite.  In general, nested $C_\alpha$ implies that $p_x$ in \eqref{eq:cr.contour} satisfies 
\[ \sup_\vartheta p_x(\vartheta) = 1. \]
According to \citet{shafer1976, shafer1987}, a plausibility contour satisfying this optimization property defines a {\em plausibility function} via
\begin{equation}
\label{eq:plaus.consonance}
\pl_x(A) = \sup_{\vartheta \in A} p_x(\vartheta), \quad A \subseteq \Theta, 
\end{equation}
which, in turn, corresponds to a consonant {\em belief function} given by 
\[ \bel_x(A) = 1 - \pl_x(A^c) = 1 - \sup_{\vartheta \in A^c} p_x(\vartheta), \quad A \subseteq \Theta. \]
Note that, intuitively, belief in an assertion need not correspond to belief against its complement, and it is easy to see that $\bel_x$ and $\pl_x$ meet this intuition, i.e., $\bel_x(A) \leq \pl_x(A)$ for all $A \subseteq \Theta$.  Given a consonant belief with plausibility function $\pl_x$, it is possible to back out the plausibility contour, i.e., $p_x(\vartheta) = \pl_x(\{\vartheta\})$; henceforth, I will use the latter expression for the contour function instead of carrying a separate notation.

Two immediate insights about the interpretation of $C_\alpha(x)$, for fixed $x$, emerge based on the derived belief and plausibility functions.  First, if a {\em plausibility region} $\Pi_\alpha(x)$ is defined as 
\begin{equation}
\label{eq:plaus.region}
\Pi_\alpha(x) = \bigl\{ \vartheta \in \Theta: \pl_x(\{\vartheta\}) > \alpha \bigr\}, \quad \alpha \in (0,1), \quad x \in \XX, 
\end{equation}
then it follows immediately from the definition of plausibility above that $C_\alpha(x)$ and $\Pi_\alpha(x)$ are the same.  In other words, each $\vartheta$ in $C_\alpha(x)$, for fixed $x$, is individually sufficiently plausible according to $\pl_x$.  Second, again from the definitions above, 
\[ \bel_x\{C_\alpha(x)\} = 1-\alpha, \]
so $C_\alpha(x)$, for fixed $x$, can be viewed as a set of sufficiently believable parameter values.  Note that these two assessments of $C_\alpha(x)$ based on $\bel_x$ and $\pl_x$ agree with those practical ``sufficiently believable/plausible'' interpretations in Section~\ref{S:intro}.  There, an additional adjective---``justifiably''---was used, and this is based on the coverage probability property \eqref{eq:coverage} of the confidence region.  That is, my $1-\alpha$ belief in $C_\alpha(x)$ is justified on the basis that the method used to construct it is reliable in the sense of \eqref{eq:coverage}.  This is in line with the {\em reliabilist} perspective 
in epistemology \citep[e.g.,][]{goldman1979}.   

The belief function described above is determined by the given family of confidence regions $\{C_\alpha: \alpha \in (0,1)\}$.  Since the confidence region is closely tied to the coverage probability condition \eqref{eq:coverage}, it is reasonable to ask how that condition might look in terms of its corresponding belief/plausibility function.  Certainly, there are belief functions that fail to meet this condition, so an answer to this question will provide insight as to which belief functions are consistent with the classical notion of confidence.  In terms of the contour function in \eqref{eq:cr.contour}, it is easy to check that the coverage probability condition \eqref{eq:coverage} is equivalent to 
\begin{equation}
\label{eq:contour.valid}
\sup_{\theta \in \Theta} \prob_{X|\theta}\bigl\{ \pl_X(\{\theta\}) \leq \alpha \bigr\} \leq \alpha, \quad \forall \; \alpha \in (0,1). 
\end{equation}
That is, $\pl_X(\{\theta\})$ should be stochastically no smaller than $\unif(0,1)$ under $\prob_{X|\theta}$.  This is reminiscent of the familiar (right-of-)uniform null distribution of p-values in the hypothesis testing context.  This is as far as the classical analysis goes, drawing a connection between coverage of confidence regions and Type~I error of the corresponding significance test.  Here I will take this analysis further to separate from the classical theory and demonstrate the benefit of this extended point of view.

Consonance of the derived belief/plausibility function implies something stronger than \eqref{eq:contour.valid}.  Indeed, by the definition of $\pl_x$ in \eqref{eq:plaus.consonance} and the property \eqref{eq:contour.valid}, it follows that 
\begin{equation}
\label{eq:plaus.valid}
\sup_{\theta \in A} \prob_{X|\theta} \{\pl_X(A) \leq \alpha\} \leq \alpha, \quad \forall \; A \subseteq \Theta, \quad \forall \; \alpha \in (0,1). 
\end{equation}
This means that $\pl_X(A)$ will tend to be not small, as a function of $X \sim \prob_{X|\theta}$, when the assertion $A$ is true in the sense that $\theta \in A$.  Since the above property holds for all assertions $A \subseteq \Theta$ and $\bel_x(A) = 1-\pl_x(A^c)$, a similar statement can be made in terms of the belief function, i.e., 
\[ \sup_{\theta \not\in A} \prob_{X|\theta} \{\bel_X(A) \geq 1-\alpha\} \leq \alpha. \]
I will say that a belief/plausibility function is {\em valid} if it satisfies \eqref{eq:plaus.valid}.  

That consonance allows for an extension of the basic property \eqref{eq:contour.valid} for singletons to the general validity property has some important consequences in the context of marginalization.  Indeed, if $\phi=\phi(\theta)$ is an interest parameter, taking values in $\Phi = \phi(\Theta)$, then a {\em marginal} plausibility function for $\phi$ can be defined as 
\begin{equation}
\label{eq:mpl}
\mpl_x(B) = \pl_x(\{\vartheta: \phi(\vartheta) \in B\}), \quad B \subseteq \Phi,
\end{equation}
and the corresponding marginal belief function is $\mbel_x(B) = 1 - \mpl_x(B^c)$.  That is, the marginal belief assigned to an assertion $B$ about the interest parameter $\phi$ is just the $\bel_x$-belief assigned to the assertion $\{\vartheta: \phi(\vartheta) \in B\}$ about the full parameter $\theta$.  Note that consonance implies marginalization is achieved by optimization, 
\[ \mpl_x(B) = \sup_{\vartheta: \phi(\vartheta) \in B} \pl_x(\{\vartheta\}), \]
and, in particular, a marginal plausibility contour obtains by taking $B$ a singleton, i.e., 
\[ \mpl_x(\{\varphi\}) = \sup_{\vartheta: \phi(\vartheta) = \varphi} \pl_x(\{\vartheta\}), \quad \varphi \in \Phi. \]
Since the validity property for $\bel_x$ covers all assertions about $\theta$, it immediately follows for $\mbel_x$.  That is, \eqref{eq:plaus.valid} implies 
\[ \sup_{\theta: \phi(\theta) \in B} \prob_{X|\theta} \{ \mpl_X(B) \leq \alpha \} \leq \alpha, \quad \forall\; B \subseteq \Phi, \quad \forall \; \alpha \in (0,1). \]
In particular, this implies that the marginal plausibility region 
\begin{equation}
\label{eq:mpl.region}
\Pi_\alpha^\phi(x) = \{\varphi \in \Phi: \mpl_x(\{\varphi\}) > \alpha\} 
\end{equation}
is a nominal $100(1-\alpha)$\% confidence region for $\phi$.  This analysis and conclusions hold for general belief/plausibility functions on $\Theta$ satisfying the validity property \eqref{eq:plaus.valid}; see Section~\ref{S:source}.  However, in the present case where $\bel_x$ is derived from a given family of confidence regions $C_\alpha$, it can be verified that the region $\Pi_\alpha^\phi(x)$ is exactly the ``naive'' marginal confidence region obtained by projecting $C_\alpha(x)$ to the $\phi$ margin, i.e., 
\[ \{\varphi: \phi(\vartheta) = \varphi \text{ for some } \vartheta \in C_\alpha(x)\}. \]
In this case, that the coverage property is preserved is no surprise, but it reveals that the belief/plausibility function representation and operations are natural for expressing and manipulating confidence \citep{balch2012}.  For comparison, recall from Section~\ref{SS:not.prob} that confidence properties may not be preserved under probability-style marginalization via integration.  Therefore, as mentioned in Section~\ref{S:intro}, belief and plausibility is a more appropriate framework in which to express confidence than ordinary probability.

\section{Constructing valid belief functions}
\label{S:source}


The previous section demonstrated that confidence is most appropriately expressed in terms of valid belief functions, and manipulated accordingly.  There, the belief function was derived from a given confidence region and its validity was a consequence of the assumed coverage probability feature \eqref{eq:coverage}.  Here, to build on those insights, we consider the question of how to directly construct a valid belief function without a confidence region to start with.  To my knowledge, the only available construction is via the so-called {\em inferential model} (IM) framework.  \citet{imbasics} present the following three-step IM construction, whose key feature is the introduction of a suitably calibrated random set on a specified auxiliary variable space.  

\begin{astep}
Associate data $X$ and parameter $\theta$ with an auxiliary variable $U$, taking values in $\UU$, with \emph{known} distribution $\prob_U$.  In particular, let 
\begin{equation}
\label{eq:basic.assoc}
X = a(\theta, U), \quad U \sim \prob_U. 
\end{equation}
This is just a mathematical description of an algorithm for simulating from $\prob_{X|\theta}$, and it is not necessary to assume that data is actually generated according to this process.  The inferential role played by the association is to shift primary focus from $\theta$ to $U$.  To see this, define the sets 
\begin{equation}
\label{eq:basic.focal}
\Theta_x(u) = \{\theta: x = a(\theta,u)\}, \quad u \in \UU, \quad x \in \XX. 
\end{equation} 
Given data $X=x$, if an {\em oracle} tells me a value $u^\star$ that satisfies $x=a(\theta,u^\star)$, then my inference is $\theta \in \Theta_x(u^\star)$, the ``strongest possible'' conclusion.  
\end{astep}

\begin{pstep}
Predict the unobserved value of $U$ in \eqref{eq:basic.assoc} with a random set $\S$.  This step is motivated by the shift of focus from $\theta$ to $U$, and is the feature that distinguishes the IM framework from fiducial.  The distribution $\prob_\S$ of $\S$ is to be chosen by the data analyst, subject to certain conditions; see Theorem~\ref{thm:valid} below. 
\end{pstep}

\begin{cstep}
Combine the association map in \eqref{eq:basic.focal} and the observed data $X=x$ with the random set $\S$ to obtain a new random set
\begin{equation}
\label{eq:post.focal}
\Theta_x(\S) = \bigcup_{u \in \S} \Theta_x(u). 
\end{equation}
Note that $\Theta_x(\S)$ contains the true $\theta$ if (and often only if) $\S$ likewise contains the unobserved value $u^\star$ of $U$.  The IM output is the distribution of $\Theta_x(\S)$, which, if $\Theta_x(\S)$ is non-empty with $\prob_\S$-probability~1 for each $x$, can be summarized by a {\em belief} and {\em plausibility function} pair, given by 
\begin{align*}
\bel_x(A) & = \prob_\S\{\Theta_x(\S) \subseteq A\} \\
\pl_x(A) & = \prob_\S\{\Theta_x(\S) \cap A \neq \varnothing\}, \quad A \subseteq \Theta.
\end{align*}
\end{cstep}



The pair $(\bel_x, \pl_x)$ measures the user's degree of belief about $\theta$, given data $x$ and the specified sampling model.  That is, large $\bel_x(A)$ and small $\pl_x(A)$ indicate strong and weak support in $x$ for the truthfulness of ``$\theta \in A$,'' respectively; intermediate cases correspond to certain degrees of ``don't know'' \citep{dempster2008, dempster.copss}.  Ideally, this output should represent something more than just the user's degrees of belief.  Indeed, \citet{reid.cox.2014} write that ``it is unacceptable if a procedure\ldots of representing uncertain knowledge would, if used repeatedly, give systematically misleading conclusions.''  To protect against this, the IM approach insists that its belief and plausibility functions are {\em valid} in the sense of \eqref{eq:plaus.valid}, which effectively calibrates the plausibility function values, leading to a connection with classical notions of confidence as described in Section~\ref{S:confidence}.  

How can one check that the validity condition \eqref{eq:plaus.valid} holds for the IM constructed above?  The following result, given in \citet{imbasics}, shows that validity holds for a wide class of predictive random sets $\S$ introduced in the P-step.  Let $(\UU,\U)$ be the measurable space on which $\prob_U$ is defined, and assume that $\U$ contains all closed subsets of $\UU$.  

\begin{thm}
\label{thm:valid}
Suppose that the predictive random set $\S$, supported on $\SS$, with distribution $\prob_\S$, satisfies the following:
\begin{itemize}
\item[\rm P1.] The support $\SS \subset 2^\UU$ contains $\varnothing$ and $\UU$, and: \\
{\rm (a)} is closed, i.e., each $S \in \SS$ is closed and, hence, in $\U$, and \\
{\rm (b)} is nested, i.e., for any $S,S' \in \SS$, either $S \subseteq S'$ or $S' \subseteq S$.
\item[\rm P2.] The distribution $\prob_\S$ satisfies $\prob_\S\{\S \subseteq K\} = \sup_{S \in \SS: S \subseteq K} \prob_U(S)$, for each $K \subseteq \UU$.  
\end{itemize}
In addition, if $\Theta_x(\S)$ is non-empty with $\prob_\S$-probability~1 for all $x$, then the IM is valid in the sense that its plausibility function satisfies \eqref{eq:plaus.valid}.  
\end{thm}

Constructing a random set $\S$ that satisfies P1--P2 is relatively easy; see Corollary~1 in \citet{imbasics}.  Note that the nested support property P1(b) implies that the corresponding belief function is consonant, which was important to the development and understanding in Section~\ref{SS:then.what}.  The non-emptiness condition, namely, $\Theta_x(\S) \neq \varnothing$ with $\prob_\S$-probability~1, holds trivially in many examples but not universally; often some refinements are needed, e.g., auxiliary variable dimension reduction techniques \citep{imcond, immarg} and/or random set stretching \citep{leafliu2012}.  The necessary details for the present setting will be given in Section~\ref{S:characterization}.


\section{An IM characterization of confidence}
\label{S:characterization}

\subsection{IMs with a family of random sets}

The developments in the IM literature have focused primarily on the case of a single predictive random set $\S \sim \prob_\S$, but this is not the only possibility.  Here it will be advantageous to consider a family of predictive random sets indexed by the parameter space $\Theta$.  That is, consider a fixed auxiliary variable space $\UU$ and a family of supports $\{\SS_\vartheta: \vartheta \in \Theta\}$, each consisting of subsets of $\UU$ satisfying Condition~P1 of Theorem~\ref{thm:valid} for each $\vartheta$.  Then, on each support $\SS_\vartheta$, define a distribution $\prob_{\S|\vartheta}$ to satisfy Condition~P2.   

Given an association $X=a(\theta, U)$, for $U \sim \prob_U$, where $\prob_U$ is supported on $\UU$, each of the predictive random sets $\S \sim \prob_{\S|\vartheta}$ described above, for $\vartheta \in \Theta$, defines an IM, with corresponding plausibility function 
\[ \pl_x(A \mid \vartheta) = \prob_{\S|\vartheta}\{\Theta_x(\S) \cap A \neq \varnothing\}, \quad A \subseteq \Theta, \quad \vartheta \in \Theta. \]
Here, I propose to {\em fuse} these $\vartheta$-specific plausibility functions together into a single plausibility function via the following formula:
\begin{equation}
\label{eq:fuse}
\pl_x(A) = \sup_{\vartheta \in A} \pl_x(\{\vartheta\} \mid \vartheta), \quad A \subseteq \Theta. 
\end{equation}
This kind of fused plausibility function has been used informally in applications of the so-called ``local conditional IMs'' \citep[e.g.,][]{imcond, imvch, imunif}, designed to achieve a level of dimension reduction beyond that which is available via sufficiency and/or conditioning arguments, leading to exact and efficient solutions in some challenging non-regular problems.  

According to \citet{shafer1987}, the quantity defined in \eqref{eq:fuse} is a genuine plausibility function, corresponding to a consonant belief function, if
\begin{equation}
\label{eq:max}
\sup_\vartheta \pl_x(\{\vartheta\} \mid \vartheta) = 1. 
\end{equation}
Moreover, validity of the IM corresponding to this fused plausibility function follows immediately from that for the individual $\vartheta$-specific IMs implied by Theorem~\ref{thm:valid}.  

\begin{thm}
\label{thm:valid.fuse}
Suppose that \eqref{eq:max} holds.  If the $\vartheta$-specific IM is valid for each $\vartheta \in \Theta$, then the fused IM with plausibility function \eqref{eq:fuse} is also valid.  
\end{thm}

\begin{proof}
The essential observation is that, for the plausibility function $\pl_x(\cdot)$ defined in \eqref{eq:fuse}, $\pl_x(A) \leq \alpha$ implies $\pl_x(\{\theta\} \mid \theta) \leq \alpha$ for all $\theta \in A$.  Therefore, 
\[ \sup_{\theta \in A} \prob_{X|\theta}\{\pl_X(A) \leq \alpha\} \leq \sup_{\theta \in A} \prob_{X|\theta}\{ \pl_X(\{\theta\} \mid \theta) \leq \alpha\}. \]
Validity of the IM with plausibility function $\pl_x(\cdot \mid \theta)$ for each $\theta \in \Theta$ implies that the right-hand side above is no more than $\alpha$.   Since this holds for all $A \subseteq \Theta$ and all $\alpha \in (0,1)$, the claimed validity of the fused IM follows.  
\end{proof}


\subsection{A complete-class theorem}
\label{SS:class}

The goal of this section is to establish the advertised complete-class result, i.e., given a suitable confidence region, there exists a valid (fused) IM such that the corresponding plausibility region matches the given confidence region.  Recall the sampling model $X \sim \prob_{X|\theta}$ with unknown parameter $\theta \in \Theta$.  Suppose that the parameter of interest is $\phi = \phi(\theta)$, possibly vector-valued, and let $\Phi = \phi(\Theta)$ be its range.  Let $C_\alpha: \XX \to 2^\Phi$ be the rule that determines, for any specified level $\alpha \in (0,1)$, based on data $X$, a $100(1-\alpha)$\% confidence region $C_\alpha(X)$ for $\phi=\phi(\theta)$.  Its defining property is that $C_\alpha(X)$, a random set as a function of $X \sim \prob_{X|\theta}$, satisfies the coverage probability condition \eqref{eq:coverage}.
I will also assume that the collection $\{C_\alpha: \alpha \in (0,1)\}$ is nested in the sense that $C_\alpha(x) \subseteq C_{\alpha'}(x)$ for $\alpha \geq \alpha'$ and for all $x \in \XX$, and the following limiting properties hold:
\[ \bigcup_\alpha C_\alpha(x) = \Phi \quad \text{and} \quad \bigcap_\alpha C_\alpha(x) \neq \varnothing. \]
The latter non-emptiness condition amounts to there existing a point, say, $\hat\phi(x)$ that belongs to every confidence region $C_\alpha(x)$, which is a standard relationship between point estimators and confidence regions.  For example, the Bayesian MAP estimator belongs to all of the highest posterior density credible regions.  

Next, take an association, $X=a(\theta, U)$ where $U \sim \prob_U$, like in \eqref{eq:basic.assoc}, consistent with the posited model $\prob_{X|\theta}$.  Technically, $\prob_U$ is defined on a measurable space $(\UU, \U)$, and I will assume that the $\sigma$-algebra $\U$ is sufficiently rich that it contains all the closed subsets of $\UU$ relative to the topology ${\cal T}$ on $\UU$.  Given this association, define the collection of subsets $\Theta_x(u) = \{\vartheta: x=a(\vartheta, u)\}$ as in \eqref{eq:basic.focal}, and the new collection 
\begin{equation}
\label{eq:S.alpha}
S_\alpha(\vartheta) = \text{clos}(\{u: C_\alpha(a(\vartheta, u)) \ni \phi(\vartheta)\}), \quad (\alpha, \vartheta) \in (0,1) \times \Theta, 
\end{equation}
where $\text{clos}(B)$ denotes the closure of $B \subseteq \UU$ with respect to the topology ${\cal T}$.  
In addition to the confidence-regions-are-nested assumption, the complete-class result requires the following {\em compatibility condition} on the confidence region and the association:
\begin{equation}
\label{eq:compatible}
\bigcup_{u \in S_\alpha(\vartheta)} \Theta_x(u) \neq \varnothing \quad \forall \; (x, \vartheta, \alpha) \in \XX \times \Theta \times (0,1). 
\end{equation}
In most ``regular'' problems, it is possible to arrange the association such that $\Theta_x(u)$ is non-empty for all $u$, hence compatibility \eqref{eq:compatible} is trivial, but I will postpone detailed discussion of this condition to Section~\ref{SS:conditions}.  

Theorem~\ref{thm:complete} characterizes confidence regions as IM-based plausibility regions, forging a rigorous connection between the two; a similar result connecting p-values to IMs is given in \citet{impval}.  Some remarks on the conditions and implications of the theorem are given in Section~\ref{S:remarks}.  

\begin{thm}
\label{thm:complete}
Let $C_\alpha$ be a family of nested confidence regions for $\phi=\phi(\theta)$ that satisfies \eqref{eq:coverage}.  Suppose that the sampling model $\prob_{X|\theta}$ admits an association \eqref{eq:basic.assoc} that is compatible with the family of confidence regions in the sense that \eqref{eq:compatible} holds.  Then there exists a valid (fused) IM for $\theta$, with marginal plausibility regions $\Pi_\alpha^\phi$ as in \eqref{eq:mpl.region}, such that 
\begin{equation}
\label{eq:subset}
\Pi_\alpha^\phi(x) \subseteq C_\alpha(x) \quad \forall \; (x,\alpha) \in \XX \times (0,1). 
\end{equation}
Equality holds in \eqref{eq:subset} for a particular $\alpha$ if the coverage probability function 
\begin{equation}
\label{eq:cp.function}
\phi \mapsto \inf_{\theta: \phi(\theta) = \phi} \prob_{X|\theta}\{C_\alpha(X) \ni \phi\} 
\end{equation}
is constant equal to $1-\alpha$.  
\end{thm}

\begin{proof}
For the A-step of the IM construction, take any association $X = a(\theta, U)$, $U \sim \prob_U$, consistent with the above sampling distribution $\prob_{X|\theta}$ and compatible with the given confidence region in the sense that \eqref{eq:compatible} holds.  For the P-step, take a family of predictive random sets $\S \sim \prob_{\S|\vartheta}$, indexed by $\vartheta \in \Theta$, with support 
\[ \SS_\vartheta = \{S_\alpha(\vartheta): \alpha \in (0,1)\}, \]
closed and nested by the definition of $S_\alpha(\vartheta)$ in \eqref{eq:S.alpha}, and distribution $\prob_{\S|\vartheta}$ satisfying 
\begin{equation}
\label{eq:prob}
\prob_{\S|\vartheta}(\S \subseteq K) = \sup_{\alpha: S_\alpha(\vartheta) \subseteq K} \prob_U\{S_\alpha(\vartheta)\}. 
\end{equation}
That $S_\alpha(\vartheta)$ is $\prob_U$-measurable follows from the fact that $S_\alpha(\vartheta)$ is closed and the $\sigma$-algebra $\U$ contains all closed subsets of $\UU$.  Note, also, that $\Theta_x(\S)$ in \eqref{eq:post.focal} is non-empty with $\prob_{\S|\vartheta}$-probability~1 according to \eqref{eq:compatible}.  For the special singleton assertion $\{\vartheta\}$ about $\theta$, the C-step returns the plausibility function 
\[ \pl_x(\{\vartheta\} \mid \vartheta) = \prob_{\S|\vartheta}\{\Theta_x(\S) \ni \vartheta\} = \prob_{\S|\vartheta}\{\S \cap \UU_x(\vartheta) \neq \varnothing\}, \]
where $\UU_x(\vartheta)$ is defined as 
\[ \UU_x(\vartheta) = \{u: x=a(\vartheta, u)\}. \]
This determines the fused plausibility function \eqref{eq:fuse} and the corresponding fused IM is valid based on Theorem~\ref{thm:valid.fuse}.  To establish the desired connection between the plausibility region of this valid IM and the given confidence region, define the index 
\begin{equation}
\label{eq:alpha}
\alpha(x,\vartheta) = \sup\{\alpha \in (0,1): S_\alpha(\vartheta) \cap \UU_x(\vartheta) \neq \varnothing\}, 
\end{equation}
so that 
\[ \pl_x(\{\vartheta\} \mid \vartheta) = 1-\prob_{\S|\vartheta}\{\S \subseteq S_{\alpha(x,\vartheta)}(\vartheta)\} \leq \alpha(x,\vartheta), \]
where the inequality is due to \eqref{eq:prob} and \eqref{eq:coverage}.  The marginal plausibility function $\mpl_x$ for $\phi=\phi(\theta)$ is given in \eqref{eq:mpl}, and the corresponding marginal plausibility region $\Pi_\alpha^\phi(x)$ is in \eqref{eq:mpl.region}.  This region satisfies:
\begin{align*}
\varphi \in \Pi_\alpha^\phi(x) \iff &\sup_{\vartheta: \phi(\vartheta) = \varphi} \pl_x(\{\vartheta\} \mid \vartheta) > \alpha \\
\iff &\pl_x(\{\vartheta\} \mid \vartheta) > \alpha &\text{for some $\vartheta$ with $\phi(\vartheta) = \varphi$} \\
\implies & \alpha(x,\vartheta) > \alpha & \text{for some $\vartheta$ with $\phi(\vartheta) = \varphi$} \\
\iff & S_\alpha(\vartheta) \cap \UU_x(\vartheta) \neq \varnothing & \text{for some $\vartheta$ with $\phi(\vartheta) = \varphi$} \\
\iff & \varphi \in C_\alpha(x). 
\end{align*}
Therefore, the valid IM constructed above has plausibility regions for $\phi=\phi(\theta)$ that are contained in the given confidence regions, completing the proof of the first claim.  

For the claim about when equality holds in \eqref{eq:subset}, note that the one-sided implication ``$\Longrightarrow$'' in the above display becomes two-sided if $\pl_x(\{\vartheta\} \mid \vartheta) = \alpha(x,\vartheta)$, which is the case if the coverage probability function \eqref{eq:cp.function} is constant equal to $1-\alpha$ as in the statement of the theorem.  
\end{proof} 



\section{Remarks}
\label{S:remarks}

\subsection{On the conditions of Theorem~\ref{thm:complete}}
\label{SS:conditions}

The confidence regions considered under the formulation of Theorem~\ref{thm:complete} are those that are nested, i.e., $C_\alpha(x) \subseteq C_{\alpha'}(x)$ for $\alpha \geq \alpha'$.  This is an intuitively appealing property, but certainly one could entertain confidence regions which are not nested.  In the proof of Theorem~\ref{thm:complete}, this property was used only to ensure that the support sets $\{S_\alpha(\vartheta): \alpha \in (0,1)\}$ were also nested, which was one of the conditions imposed in Theorem~\ref{thm:valid}.  But even this latter theorem does not strictly require nested support sets.  Indeed, validity holds for more general supports, but Theorem~3 in \citet{imbasics} says that any IM based on a predictive random set with a non-nested support can be made more efficient with the choice of a nested random set.  In other words, IMs based on non-nested random sets are inadmissible in the usual decision-theoretic sense, hence my choice here to focus on nested confidence regions and the corresponding nested predictive random sets.  

The compatibility condition \eqref{eq:compatible} in Theorem~\ref{thm:complete}, on the other hand, is less transparent.  To provide some further insight into this condition, below I describe three different scenarios, some general and some more specific.  

First, consider the special case where both the statistical model and the confidence region have a certain transformation structure.  Let $\G$ be a group of transformations $g: \XX \to \XX$, with associated group $\Gbar$, such that $X \sim \prob_\theta$ implies $gX \sim \prob_{\bar g \theta}$; as is customary in this context, I write $gx$ instead of $g(x)$, and assume that $\Gbar$ acts transitively on $\Theta$.  Next, define another associated group $\Gtilde$ such that $\phi(g\theta) = \tilde g \phi(\theta)$.  Then a confidence region $C_\alpha$ for $\phi$ is {\em invariant} if $C_\alpha(gx) = \tilde C_\alpha(x)$.  \citet{arnold1984} gives further details on this setup, along with some examples.  For the present context, the key point is that the model can be described by first picking any ``baseline'' value of $\theta$, say, 0, identifying the mapping $g_\theta$ that converts $U \sim \prob_0$ to $X \sim \prob_\theta$.  Therefore, the association \eqref{eq:basic.assoc} looks like $X=g_\theta U$, with $U \sim \prob_0$, and the support sets $S_\alpha(\theta)$ are free of $\theta$, i.e., 
\[ S_\alpha(\theta) = \{u: C_\alpha(g_\theta u) \ni \phi(\theta)\} 
= \{u: C_\alpha(u) \ni \phi(0)\}. \]
Therefore, no gluing is required in the IM construction in the proof of Theorem~\ref{thm:complete}.  Moreover, if the invariance structure is considered after reducing to a minimal sufficient statistic, as in Arnold, then, e.g., $\G$ acting transitively on $\XX$ implies $\Theta_x(u) \neq \varnothing$ for all $(x,u)$, hence \eqref{eq:compatible}.  

Second, the compatibility condition is trivial whenever $\Theta_x(u)$ is non-empty for every pair $(x,u)$.  This would hold for most ``regular'' problems when $X$ represents a minimal sufficient statistic of the same dimension as $\theta$, there are no non-trivial constraints on the parameter space, and the confidence region $C_\alpha$ is a function of only the minimal sufficient statistic.  Many problems are of this type; see Section~\ref{S:examples}.  

Third, regularity is not necessary for compatibility to hold.  Suppose that $Y_1,\ldots,Y_n$ are iid $\unif(\theta,\theta+1)$, and define the minimal sufficient statistic $X=(X_1,X_2)$, the sample minimum and maximum, respectively.  Since $X$ is two-dimensional but $\theta$ is a scalar, the model is considered to be ``non-regular.''  Here I will consider a Bayesian approach with a flat prior for $\theta$; then the posterior is $\unif(X_2-1, X_1)$, and the equi-tailed $100(1-\alpha)$\% credible interval is 
\[ C_\alpha(x) = \bigl[x_1 - (1-d_x)(1-\tfrac{\alpha}{2}), \, x_1 - (1-d_x) \tfrac{\alpha}{2} \bigr], \quad d_x = x_2 - x_1, \]
and it is nested and satisfies the coverage condition \eqref{eq:coverage} with equality.  The natural association is $X=\theta + U$, where $U=(U_1,U_2)$ denotes the minimum and maximum of an iid sample of size $n$ from $\unif(0,1)$.  With this choice, 
\begin{align*}
S_\alpha(\vartheta) & = \{(u_1,u_2): [u_1 + \vartheta - (1-d_u)(1-\tfrac{\alpha}{2}), \, u_1 + \vartheta - (1-d_u)\tfrac{\alpha}{2}] \ni \vartheta \} \\
& = \{(u_1, u_2): \tfrac{\alpha}{2-\alpha} (1-u_2) \leq u_1 \leq \tfrac{2-\alpha}{\alpha} (1-u_2) \}.
\end{align*}
Note that any $(u_1,u_2)$ with $u_1 \leq u_2$ and $u_1 = 1-u_2$ belongs to $S_\alpha(\vartheta)$ for all $\alpha$ and for all $\vartheta$.  In particular, if $\hat\theta(x) = \frac12\{x_1 + (x_2-1)\}$ and $u(x)=x - \hat\theta(x) 1_2$, then 
\[ u(x) \in S_\alpha(\vartheta) \quad \text{and} \quad \Theta_x(u(x)) = \{\hat\theta(x)\} \neq \varnothing, \]
which implies \eqref{eq:compatible}.  

To summarize, the compatibility condition seems to be rather weak and, in fact, I have not been able to find a scenario where it fails.  On the other hand, however, there are many degrees of freedom in the choice of $C_\alpha$, so I have not been able to prove that compatibility always holds either.

Finally, Theorem~\ref{thm:complete} suggests that the IM's marginal plausibility regions can in some cases be proper subsets of the given confidence regions.  The binomial example in Section~\ref{SS:binomial} illustrates this.  In cases with a vector $\theta$ with, say, a scalar interest parameter $\phi=\phi(\theta)$, it is common for confidence regions $C_\alpha$ to be conservative at particular $\theta$ values.  However, strict inequality in \eqref{eq:subset} requires that $C_\alpha$ be {\em uniformly conservative} for $\theta$ having a given value of $\phi$, which is less common.  Even confidence regions that are conservative for some parameter values are typically motivated by their efficient performance for other parameter values.  For example, in the mixed-effects model setting considered by \citet{e.hannig.iyer.2008}, their simulations show that certain confidence intervals 
for the random-effect variance are conservative under some parameter settings but efficient under others.  So, a given $C_\alpha$ being conservative at a particular $\theta$ does not imply that the corresponding IM-based plausibility region will be smaller; a type of uniform conservatism is required for automatic improvement.  The Behrens--Fisher illustration in Section~\ref{SS:bf} provides some more details.


\subsection{On connections to the Bayesian literature}

A fundamental result in Bayesian analysis is the probability matching property, namely, in certain cases, the Bayesian posterior credible regions will be first- or higher-order accurate in the sense that their frequentist coverage probability will be within a suitably narrow range around the target nominal value; see \citet{datta.mukherjee.book} and \citet{mghosh2011} for a review, and see \citet{fraser2011} and \citet{fraser.etal.2016} for some limitations.  One interpretation of these results is that there exists a prior such that, asymptotically, the corresponding Bayesian credible regions will roughly agree with a standard frequentist confidence region based on, say, likelihood theory.  This is a type of complete-class theorem, similar to others that can be found in the literature \citep[e.g.,][Chap.~8]{berger1985}, generally providing some frequentist justification for certain Bayesian methods.  But having to choose a particular prior in order to get valid credible regions strips away almost all that is genuinely Bayesian about this approach, so a characterization of confidence through a Bayesian argument seems out of reach.  Compare this to  Theorems~\ref{thm:valid}--\ref{thm:valid.fuse} that give sufficient conditions for IMs to yield valid plausibility regions, and Theorem~\ref{thm:complete} that says all valid confidence regions correspond to an IM.  


\subsection{On implications for the IM framework}

Confidence regions are used in all areas of statistical applications as they provide a means for uncertainty quantification.  So being able to identify all reasonable confidence regions as corresponding to output from a valid IM has some important implications for the IM framework itself.  That is, there is no systematic bias incurred by adopting the IM approach in the sense that there is no ``good answer'' that cannot be reached through the use of IMs.  Some potential users of IMs may complain that the A- and P-steps of the IM construction require certain subjective choices, but Theorem~\ref{thm:complete} indicates that this is not a restriction: if the user would be happy with a valid Bayesian credible region, then there is an IM that would achieve that, as well as an ``algorithm'' for finding it.  It is important to emphasize, however, that the IM approach is constructive and does not rely on there being an already-known ``good answer'' to start with or compare to.  That is, one can start from a given sampling model and, through the three-steps described in Section~\ref{S:source}, directly construct a valid IM that may produce results different than existing Bayesian or frequentist solutions.  In other words, not only are the user's choices in the A- and P-steps not a restriction, they actually provide flexibility, though not so much flexibility that the user could inadvertently violate the essential validity property.  To conclude this remark, I want to emphasize that the IM approach is a general framework, a fundamental extension to Fisher's work on fiducial, and should not be viewed just as a method for constructing confidence regions or other frequentist procedures.

\section{Examples}
\label{S:examples}


\subsection{Binomial problem}
\label{SS:binomial}

Let $X \sim \bin(n,\theta)$, where the size $n$ is known but the success probability $\theta \in [0,1]$ is unknown.  A standard confidence interval for $\theta$, satisfying the coverage probability condition \eqref{eq:coverage}, is the Clopper--Pearson interval
\[ C_\alpha(x) = \{\theta: F_\theta(x) \geq \tfrac{\alpha}{2}, \, 1-F_\theta(x-1) \geq \tfrac{\alpha}{2}\}, \quad \alpha \in (0,1), \]
where $F_\theta$ is the binomial distribution function.  It is well-known \citep[e.g.,][Figure~11]{bcd2001} that the coverage probability of $C_\alpha$ varies wildly as a function of $\theta$ so it does not have the nice ``uniform exactness'' properties that pivot-based intervals enjoy.  Therefore, the IM plausibility region derived using the strategy in the proof of Theorem~\ref{thm:complete} may provide some improvement to the Clopper--Pearson interval.  

To start, I will take the association in \eqref{eq:basic.assoc} as 
\[ F_\theta(X-1) < U \leq F_\theta(X), \quad U \sim \unif(0,1). \]
This is a standard formula for simulating a binomial variate, and I write the solution $X$ as $F_\theta^{-1}(U)$, which can be computed via the {\tt qbinom} function in R.  Then the support set $S_\alpha(\theta)$ is given by 
\[ S_\alpha(\theta) = \{u: F_\theta(F_\theta^{-1}(u)) \geq \tfrac{\alpha}{2}, \, 1-F_\theta(F_\theta^{-1}(u)-1) \geq \tfrac{\alpha}{2}\}. \]
The set $\UU_x(\theta) = \{u: F_\theta(x-1) < u \leq F_\theta(x)\}$ is an interval, and I need the supremum of all $\alpha$ such that $S_\alpha(\theta) \cap \UU_x(\theta) \neq \varnothing$.  This value is the index $\alpha(x,\theta)$ in \eqref{eq:alpha} which, in this case, is given by 
\[ \alpha(x,\theta) = \begin{cases}
2\{1 - F_\theta(x-1)\}, & \text{if $F_\theta(x-1) \geq \frac12$} \\
2 F_\theta(x), & \text{if $F_\theta(x) \leq \frac12$} \\
1, & \text{if $F_\theta(x-1) < \frac12 < F_\theta(x)$}.
\end{cases} \]
After adjusting for the non-zero width of the smallest $S_{\alpha(x,\theta)}(\theta)$, the fused IM plausibility function is 
\[ \pl_x(\{\theta\} \mid \theta) = \begin{cases}
1, & \text{if $F_\theta(x-1) < \frac12 < F_\theta(x)$} \\
1 - g(\theta,x) , & \text{otherwise}.
\end{cases} \]
where
\[ g(\theta, x) = \prob_U\{F_\theta(F_\theta^{-1}(U)) > \tfrac{\alpha(x,\theta)}{2}, \, F_\theta(F_\theta^{-1}(U)-1) < 1-\tfrac{\alpha(x,\theta)}{2}\}, \]
which can be readily computed numerically, with or without Monte Carlo. 

As an illustration, consider an experiment with $n=25$ trials and $X=17$ observed successes.  Figure~\ref{fig:binom} shows a plot of the plausibility contour extracted directly from the Clopper--Pearson confidence interval, which is actually just $\alpha(x,\theta)$, and that for the fused IM based on the construction in the proof of Theorem~\ref{thm:complete}.  Note that the IM-based plausibility contour is no wider than that from Clopper--Pearson, which demonstrates the potential improvement from the particular IM construction.  Finally, I mention that this is not the only IM-based plausibility interval available for this binomial problem; e.g., \citet{plausfn} derives one from a likelihood-based IM.

\begin{figure}[t]
\begin{center}
\scalebox{0.70}{\includegraphics{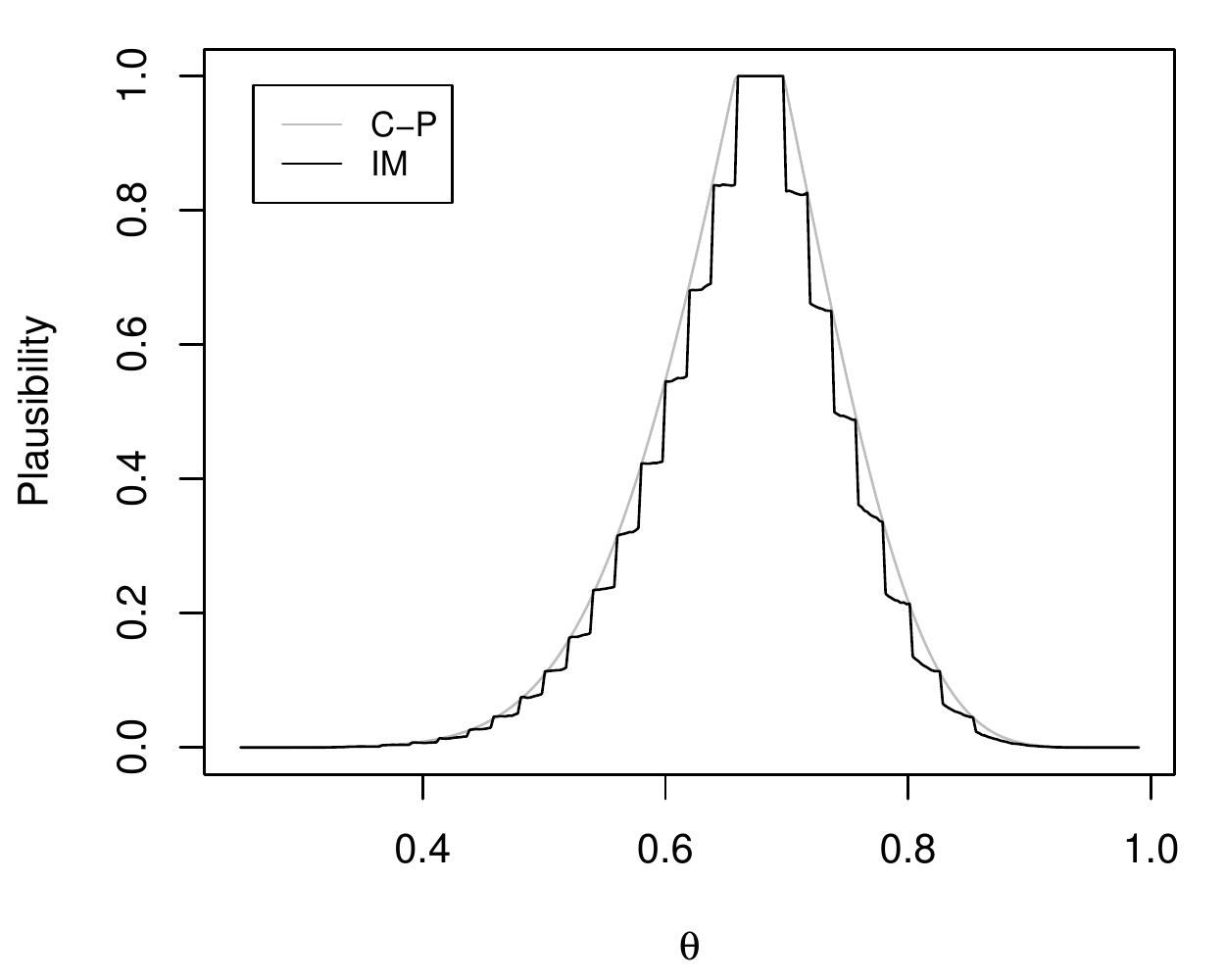}}
\end{center}
\caption{Plot of the fust IM and the Clopper--Pearson (C--P) plausibility contours for the binomial problem, with $n=25$ and $X=17$.}
\label{fig:binom}
\end{figure}

\subsection{Behrens--Fisher problem}
\label{SS:bf}

Consider two independent data sets $Y_{11},\ldots,Y_{1n_1}$ and $Y_{21},\ldots,Y_{2n_2}$ from $\nm(\mu_1, \sigma_1^2)$ and $\nm(\mu_2, \sigma_2^2)$, respectively, where $\theta=(\mu_1,\mu_2,\sigma_1^2, \sigma_2^2)$ is unknown, but the goal is inference on $\phi(\theta) = \mu_1-\mu_2$.  Let $X=(M_1, M_2, V_1, V_2)$ denote the minimal sufficient statistic, consisting of the two sample means and two sample variances, respectively.  Then the famous Hsu--Scheff\'e interval \citep{hsu1938, scheffe1970} is given by 
\[ C_\alpha(x) = (m_1 - m_2) \pm t_{\alpha,n}^\star \, f(v_1, v_2), \]
where $f(v_1, v_2) = (v_1/n_1 + v_2/n_2)^{1/2}$ and $t_{\alpha,n}^\star$ is the $1-\frac{\alpha}{2}$ quantile of a Student-t distribution with $\min\{n_1,n_2\} - 1$ degrees of freedom.  It follows from results in, e.g.,  \citet{mickey.brown.1966} that the above interval achieves the coverage probability condition \eqref{eq:coverage}, but could be somewhat conservative for small $n$ and/or certain $\theta$ configurations.  Here I derive the plausibility counterpart for $C_\alpha$ as in the proof of Theorem~\ref{thm:complete}.  

For the association \eqref{eq:basic.assoc}, I will take 
\[ D = \phi + f(\sigma_1^2, \sigma_2^2) \, U_1, \quad V_k = \sigma_k^2 U_{2k}, \quad k=1,2, \]
where $D=M_1-M_2$ and the auxiliary variable $U=(U_1,U_{21},U_{22})$ consists of independent components with $U_1 \sim \nm(0,1)$ and $U_{2k} \sim \chisq(n_k-1) / (n_k - 1)$, $k=1,2$.  Since 
\[ \frac{D - \phi}{f(V_1,V_2)} = \frac{f(\sigma_1^2, \sigma_2^2)}{f(\sigma_1^2 U_{21}, \sigma_2^2 U_{22})} \, U_1, \]
the support sets $S_\alpha(\theta)$ can be written as  
\[ S_\alpha(\theta) = \Bigl\{ u=(u_1, u_{21}, u_{22}): \frac{|u_1|}{\{\lambda_\theta u_{21} + (1-\lambda_\theta) u_{22}\}^{1/2}} \leq t_{\alpha, n}^\star \Bigr\}, \]
where $\lambda_\theta = \{1 + (n_1 \sigma_1^2)/(n_2 \sigma_2^2)\}^{-1}$, which takes values in $[0,1]$ as a function of $\theta$.  Note that the set $S_\alpha(\theta)$ depends on $\theta$ only through $\lambda_\theta$, not on $\phi$ or on the specific values of $\sigma_1^2$ and $\sigma_2^2$.  Then the index $\alpha(x,\theta)$ in \eqref{eq:alpha} is given by 
\[ \alpha(x,\theta) = 2 \bigl| F_n\bigl( t(x,\theta) \bigr) - 1 \bigr|, \]
where $t(x,\theta) = (d - \phi)/f(v_1,v_2)$ and $F_n$ is the Student-t distribution function with $\min\{n_1,n_2\}-1$ degrees of freedom, which actually only depends on $\phi$.  Then the fused IM plausibility function satisfies 
\[ \pl_x(\{\theta\} \mid \theta) = \prob_U\{S_{\alpha(x,\theta)}(\theta)\}, \]
which can be readily evaluated via Monte Carlo; then the corresponding marginal plausibility function for $\phi$ in \eqref{eq:mpl} can be obtained by optimization over $\lambda$.  

For illustration, consider the example in \citet[][p.~83]{lehmann1975} on travel times for two different routes; summary statistics are:
\begin{align*}
n_1 & = 5 & m_1 & = 7.580 & v_1 & = 2.237 \\
n_2 & = 11 & m_2 & = 6.136 & v_2 & = 0.073.
\end{align*}
The goal is to compare the mean travel times for two different routes.  Figure~\ref{fig:bf} shows a plot of the plausibility contour $\alpha(x,\theta)$ for $\phi$ extracted directly from the Hsu--Scheff\'e confidence interval, the plausibility contour above for a range of $\lambda$, and the corresponding marginal plausibility based on optimizing over these $\lambda$.  As expected, the $\lambda$-dependent plausibilities (gray lines) are individually more efficient than that of Hsu--Scheff\'e.  However, marginalizing over $\lambda$ via optimization widens the plausibility contours to agree exactly with Hsu--Scheff\'e, as predicted by Theorem~\ref{thm:complete}.  

\begin{figure}[t]
\begin{center}
\scalebox{0.70}{\includegraphics{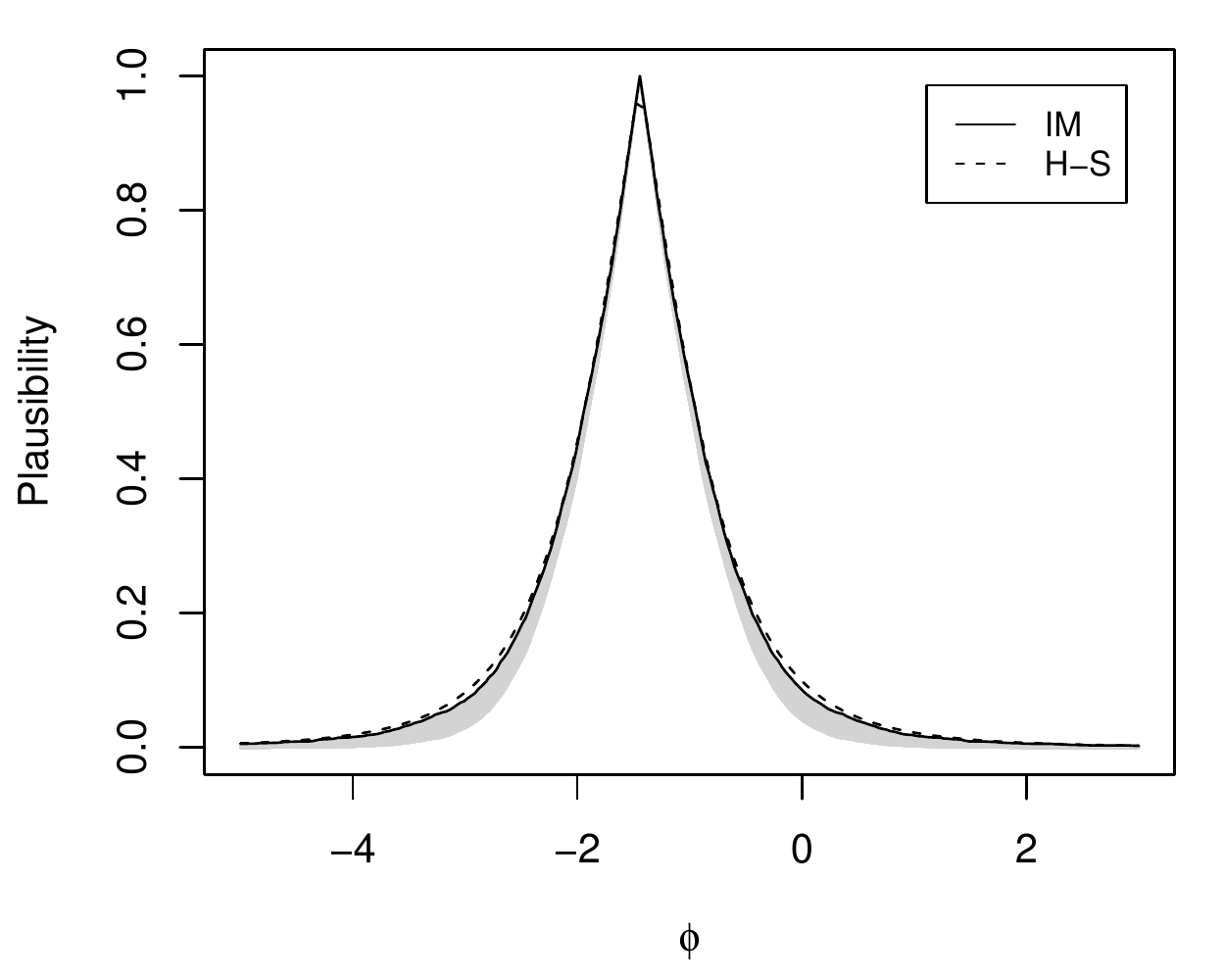}}
\end{center}
\caption{Plausibility functions for the Hsu--Scheff\'e interval (H--S) and for the derived IM in the travel times illustration of the Behrens--Fisher problem. Differences between the solid and dashed lines are the result of optimizing over only a finite range of $\lambda$.}
\label{fig:bf}
\end{figure}

\subsection{A nonparametric problem}

Although the developments in the previous sections focus primarily on finite-dimensional or parametric models, there is no reason that they cannot be applied to infinite-dimensional or nonparametric problems just the same.  As a simple illustration, consider real-valued data $X=(X_1,\ldots,X_n)$ iid with distribution function $F$, where the goal is inference on the ``parameter'' $F$.  A standard nonparametric confidence region is based on the Dvoretsky--Kiefer--Wolfowitz inequality \citep[DKW,][]{dkw1956} and corresponds to a ball with respect to sup-norm $\|\cdot\|_\infty$, i.e., 
\[ C_\alpha(x) = \{F: \|\hat F_x - F\|_\infty \leq \delta_{n,\alpha}\}, \]
where $\hat F_x$ denotes the empirical distribution function based on a sample $x=(x_1,\ldots,x_n)$, and $\delta_{n,\alpha} = \{\log(2/\alpha) / (2n)\}^{1/2}$.  Equivalently, $C_\alpha(x)$ can be viewed as a confidence band, giving pointwise lower and upper confidence limits; see Figure~\ref{fig:np} below.  Since the DKW inequality upon which the coverage probability results for $C_\alpha$ are based holds for all $F$, there will be some $F$ at which the coverage is conservative, so there is reason to expect that the IM-based confidence region derived from this will be smaller and less conservative, similar to the binomial case presented above.  

To keep things simple here, I will write $F^{-1}$ for the inverse of a distribution function $F$, and ignore the possible non-uniqueness due to $F$ flattening out on some interval.  Then a natural association is 
\[ X_i = F^{-1}(U_i), \quad i=1,\ldots,n, \]
where $U=(U_1,\ldots,U_n)$ are iid $\unif(0,1)$.  Writing this in vectorized form, i.e., $X=F^{-1}(U)$, it follows that 
\[ S_\alpha(F) = \{u: C_\alpha(F^{-1}(u)) \ni F\} = \{u: \|\hat F_{F^{-1}(u)} - F\|_\infty \leq \delta_{n,\alpha}\}, \]
and it is easy to confirm that $S_\alpha(F) \equiv S_\alpha$ does not depend on $F$, hence there is no need for a fused IM.  Furthermore, the index $\alpha(x,F)$ defined in the proof of Theorem~\ref{thm:complete} is 
\[ \alpha(x,F) = \min\bigl\{ 1, 2 e^{-2n \|\hat F_x - F\|_\infty^2} \bigr\}. \]
Therefore, the plausibility contour for a singleton assertion $\{F\}$ is 
\[ \pl_x(\{F\}) = 1-\prob_U\{S_{\alpha(x,F)}\} \]
which can readily be evaluated via Monte Carlo for any fixed $F$.  

For an illustration, consider the nerve data set from \citet{cox.lewis.book}, analyzed in \citet[][Example~2.1]{wasserman2006book} on $n=799$ waiting times between pulses along a nerve fiber.  A plot of the empirical distribution function along with the lower and upper 95\% confidence bands based on the DKW inequality is shown in Figure~\ref{fig:np}.  While it is difficult to draw a new pair of confidence bands corresponding to IM-based plausibility analysis, it is simple to verify the claim that the bands would indeed be narrower.  Take the lower bound in Figure~\ref{fig:np} as a candidate $F$; then $\alpha(x,F) \approx 0.05$ but $\pl_x(\{F\}) \approx 0$ and, therefore, the 95\% lower confidence bound based on the DKW inequality is not included in the 95\% plausibility region.

\begin{figure}[t]
\begin{center}
\scalebox{0.70}{\includegraphics{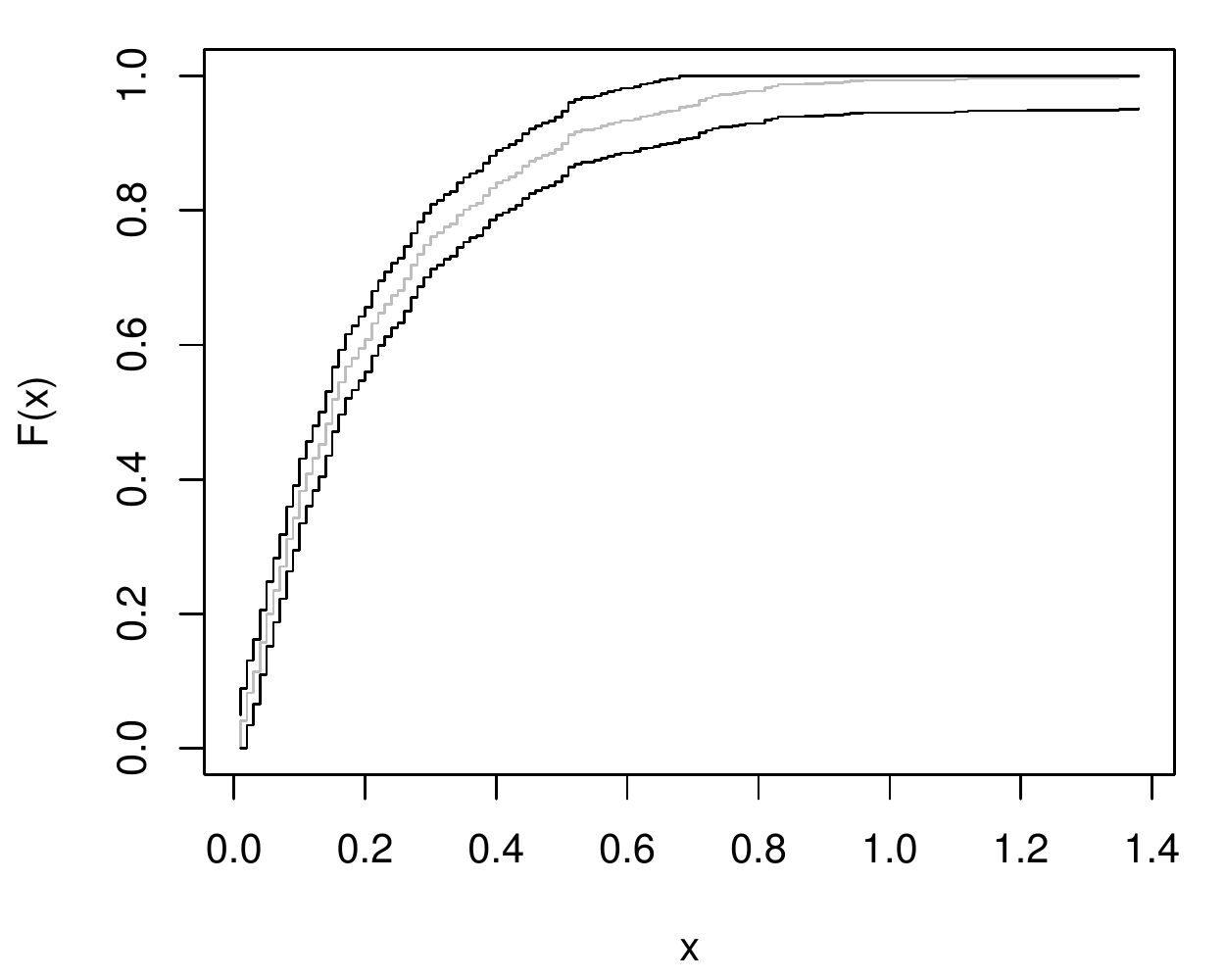}}
\end{center}
\caption{Empirical distribution function (gray) and lower and upper 95\% confidence bands (black) for the nerve data example.}
\label{fig:np}
\end{figure}

\section{Discussion}
\label{S:discuss}

Confidence is one of the most fundamental and widely used concepts in all of statistics, and while the majority of statisticians would agree that confidence is different from probability, there seems to be no agreement about what specifically confidence is.  The goal of this paper was to provide a rigorous characterization of confidence via suitable belief/plausibility functions.  This goal was achieved by showing that every suitable confidence region corresponds to the plausibility region of a valid IM.  This characterization provides at least two important contributions.  

First, statisticians have tried (e.g., ASA statement on p-values, etc) to drive home the standard/textbook interpretation of confidence, but without much success.  It is easy for statisticians to complain that practitioners ``just don't get it,'' but it would be more productive to consider the possibility that this textbook explanation is not clear or otherwise not fully satisfactory.  The fact is, despite our efforts to fit confidence into a familiar framework, probability simply does not give an adequate characterization.  ``Plausibility intervals'' is a natural and intuitively clear way to describe confidence intervals to non-statisticians, and there is now no reason to shy away from this explanation since the present paper provides a rigorous justification for it.  

Second, it establishes that the IM framework is more than just an alternative to Bayes, frequentist, fiducial, etc.  At the recent {\em BFF4} meeting at Harvard University---see \citet{xl.bff.2017} for details, including explanation of the name---a part of Nancy Reid's invited talk summarized each of the different approaches discussed at the meeting, with a one-line quote from the respective advocates of those approaches.  For the IM framework, Professor Reid summarized my position by ``IMs are the only answer,'' and the results in the present paper provide a justification for this point of view: if there is a good answer, then there is an IM solution that is the same or better.



\bibliographystyle{apalike}
\bibliography{/Users/rgmarti3/Dropbox/Research/mybib}

\begin{thebibliography}{}

\bibitem[Arnold, 1984]{arnold1984}
Arnold, S.~F. (1984).
\newblock Pivotal quantities and invariant confidence regions.
\newblock {\em Statist. Decisions}, 2(3-4):257--280.

\bibitem[Balch, 2012]{balch2012}
Balch, M.~S. (2012).
\newblock Mathematical foundations for a theory of confidence structures.
\newblock {\em Internat. J. Approx. Reason.}, 53(7):1003--1019.

\bibitem[Balch et~al., 2017]{balch.martin.ferson.2017}
Balch, M.~S., Martin, R., and Ferson, S. (2017).
\newblock Coverage probability fails to ensure reliable inference.
\newblock Unpublished manuscript, {\tt arXiv:1706.08565}.

\bibitem[Berger, 1985]{berger1985}
Berger, J.~O. (1985).
\newblock {\em Statistical Decision Theory and {B}ayesian Analysis}.
\newblock Springer-Verlag, New York, second edition.

\bibitem[Birnbaum, 1961]{birnbaum1961}
Birnbaum, A. (1961).
\newblock Confidence curves: an omnibus technique for estimation and testing
  statistical hypotheses.
\newblock {\em J. Amer. Statist. Assoc.}, 56:246--249.

\bibitem[Blaker and Spj{\o}tvoll, 2000]{blaker.spjotvoll.2000}
Blaker, H. and Spj{\o}tvoll, E. (2000).
\newblock Paradoxes and improvements in interval estimation.
\newblock {\em Amer. Statist.}, 54(4):242--247.

\bibitem[Brown et~al., 2001]{bcd2001}
Brown, L.~D., Cai, T.~T., and DasGupta, A. (2001).
\newblock Interval estimation for a binomial proportion.
\newblock {\em Statist. Sci.}, 16(2):101--133.
\newblock With comments and a rejoinder by the authors.

\bibitem[Cheng et~al., 2014]{imvch}
Cheng, Q., Gao, X., and Martin, R. (2014).
\newblock Exact prior-free probabilistic inference on the heritability
  coefficient in a linear mixed model.
\newblock {\em Electron. J. Stat.}, 8(2):3062--3076.

\bibitem[Cox and Lewis, 1966]{cox.lewis.book}
Cox, D.~R. and Lewis, P. A.~W. (1966).
\newblock {\em The Statistical Analysis of Series of Events}.
\newblock Methuen \& Co., Ltd., London; John Wiley \& Sons, Inc., New York.

\bibitem[Creasy, 1954]{creasy1954}
Creasy, M.~A. (1954).
\newblock Symposium on interval estimation: {L}imits for the ratio of means.
\newblock {\em J. Roy. Statist. Soc. Ser. B.}, 16:186--194.

\bibitem[Datta and Mukerjee, 2004]{datta.mukherjee.book}
Datta, G.~S. and Mukerjee, R. (2004).
\newblock {\em Probability Matching Priors: Higher Order Asymptotics}, volume
  178 of {\em Lecture Notes in Statistics}.
\newblock Springer-Verlag, New York.

\bibitem[Dempster, 2008]{dempster2008}
Dempster, A.~P. (2008).
\newblock The {D}empster--{S}hafer calculus for statisticians.
\newblock {\em Internat. J. Approx. Reason.}, 48(2):365--377.

\bibitem[Dempster, 2014]{dempster.copss}
Dempster, A.~P. (2014).
\newblock Statistical inference from a {D}empster--{S}hafer perspective.
\newblock In Lin, X., Genest, C., Banks, D.~L., Molenberghs, G., Scott, D.~W.,
  and Wang, J.-L., editors, {\em Past, Present, and Future of Statistical
  Science}, chapter~24. Chapman \& Hall/CRC Press.

\bibitem[Dufour, 1997]{dufour1997}
Dufour, J.-M. (1997).
\newblock Some impossibility theorems in econometrics with applications to
  structural and dynamic models.
\newblock {\em Econometrica}, 65(6):1365--1387.

\bibitem[Dvoretzky et~al., 1956]{dkw1956}
Dvoretzky, A., Kiefer, J., and Wolfowitz, J. (1956).
\newblock Asymptotic minimax character of the sample distribution function and
  of the classical multinomial estimator.
\newblock {\em Ann. Math. Statist.}, 27:642--669.

\bibitem[E et~al., 2008]{e.hannig.iyer.2008}
E, L., Hannig, J., and Iyer, H. (2008).
\newblock Fiducial intervals for variance components in an unbalanced
  two-component normal mixed linear model.
\newblock {\em J. Amer. Statist. Assoc.}, 103(482):854--865.

\bibitem[Efron, 1998]{efron1998}
Efron, B. (1998).
\newblock R. {A}. {F}isher in the 21st century.
\newblock {\em Statist. Sci.}, 13(2):95--122.

\bibitem[Ermini~Leaf and Liu, 2012]{leafliu2012}
Ermini~Leaf, D. and Liu, C. (2012).
\newblock Inference about constrained parameters using the elastic belief
  method.
\newblock {\em Internat. J. Approx. Reason.}, 53(5):709--727.

\bibitem[Fieller, 1954]{fieller1954}
Fieller, E.~C. (1954).
\newblock Symposium on interval estimation: {S}ome problems in interval
  estimation.
\newblock {\em J. Roy. Statist. Soc. Ser. B.}, 16:175--185.

\bibitem[Fisher, 1973]{fisher1973}
Fisher, R.~A. (1973).
\newblock {\em Statistical Methods and Scientific Inference}.
\newblock Hafner Press, New York, 3rd edition.

\bibitem[Fraser, 1990]{fraser1990}
Fraser, D. A.~S. (1990).
\newblock Tail probabilities from observed likelihoods.
\newblock {\em Biometrika}, 77(1):65--76.

\bibitem[Fraser, 1991]{fraser1991}
Fraser, D. A.~S. (1991).
\newblock Statistical inference: likelihood to significance.
\newblock {\em J. Amer. Statist. Assoc.}, 86(414):258--265.

\bibitem[Fraser, 2011a]{fraser2011}
Fraser, D. A.~S. (2011a).
\newblock Is {B}ayes posterior just quick and dirty confidence?
\newblock {\em Statist. Sci.}, 26(3):299--316.

\bibitem[Fraser, 2011b]{fraser2011.rejoinder}
Fraser, D. A.~S. (2011b).
\newblock Rejoinder: ``{I}s {B}ayes posterior just quick and dirty
  confidence?''.
\newblock {\em Statist. Sci.}, 26(3):329--331.

\bibitem[Fraser, 2014]{fraser.copss}
Fraser, D. A.~S. (2014).
\newblock Why does statistics have two theories?
\newblock In Lin, X., Genest, C., Banks, D.~L., Molenberghs, G., Scott, D.~W.,
  and Wang, J.-L., editors, {\em Past, Present, and Future of Statistical
  Science}, chapter~22. Chapman \& Hall/CRC Press.

\bibitem[Fraser et~al., 2016]{fraser.etal.2016}
Fraser, D. A.~S., B\'edard, M., Wong, A., Lin, W., and Fraser, A.~M. (2016).
\newblock Bayes, reproducibility and the quest for truth.
\newblock {\em Statist. Sci.}, 31(4):578--590.

\bibitem[Ghosh, 2011]{mghosh2011}
Ghosh, M. (2011).
\newblock Objective priors: An introduction for frequentists.
\newblock {\em Statist. Sci.}, 26(2):187--202.

\bibitem[Gleser and Hwang, 1987]{gleser.hwang.1987}
Gleser, L.~J. and Hwang, J.~T. (1987).
\newblock The nonexistence of {$100(1-\alpha)\%$} confidence sets of finite
  expected diameter in errors-in-variables and related models.
\newblock {\em Ann. Statist.}, 15(4):1351--1362.

\bibitem[Goldman, 1979]{goldman1979}
Goldman, A.~I. (1979).
\newblock Reliabilism: What is justified belief?
\newblock In Pappas, G.~S., editor, {\em Justification and Knowledge}, page~11.
  Reidel, Dordrecht, Holland.

\bibitem[Hannig et~al., 2016]{hannig.review}
Hannig, J., Iyer, H., Lai, R. C.~S., and Lee, T. C.~M. (2016).
\newblock Generalized fiducial inference: a review and new results.
\newblock {\em J. Amer. Statist. Assoc.}, 111(515):1346--1361.

\bibitem[Hinkley, 1969]{hinkley1969}
Hinkley, D.~V. (1969).
\newblock On the ratio of two correlated normal random variables.
\newblock {\em Biometrika}, 56:635--639.

\bibitem[Hsu, 1938]{hsu1938}
Hsu, P.~L. (1938).
\newblock Contributions to the theory of ``{S}tudent's'' $t$-test as applied to
  the problem of two samples.
\newblock In {\em Statistical Research Memoirs}, pages 1--24. University
  College, London.

\bibitem[Ionides et~al., 2017]{ionides.response}
Ionides, E.~L., Giessing, A., Ritov, Y., and Page, S.~E. (2017).
\newblock Response to the {ASA}'s {S}tatement on {$p$}-{V}alues: {C}ontext,
  {P}rocess, and {P}urpose.
\newblock {\em Amer. Statist.}, 71(1):88--89.

\bibitem[Lehmann, 1975]{lehmann1975}
Lehmann, E.~L. (1975).
\newblock {\em Nonparametrics: Statistical Methods Based on Ranks}.
\newblock Holden-Day Inc., San Francisco, Calif.

\bibitem[Martin, 2015]{plausfn}
Martin, R. (2015).
\newblock Plausibility functions and exact frequentist inference.
\newblock {\em J. Amer. Statist. Assoc.}, 110(512):1552--1561.

\bibitem[Martin, 2016]{pvalue.course}
Martin, R. (2016).
\newblock A statistical inference course based on p-values.
\newblock \emph{Amer. Statist.}, to appear; {\tt arXiv:1606.02352}.

\bibitem[Martin and Lin, 2016]{imunif}
Martin, R. and Lin, Y. (2016).
\newblock Exact prior-free probabilistic inference in a class of non-regular
  models.
\newblock {\em Stat}, 5:312--321.

\bibitem[Martin and Liu, 2013]{imbasics}
Martin, R. and Liu, C. (2013).
\newblock Inferential models: a framework for prior-free posterior
  probabilistic inference.
\newblock {\em J. Amer. Statist. Assoc.}, 108(501):301--313.

\bibitem[Martin and Liu, 2014]{impval}
Martin, R. and Liu, C. (2014).
\newblock A note on p-values interpreted as plausibilities.
\newblock {\em Statist. Sinica}, 24(4):1703--1716.

\bibitem[Martin and Liu, 2015a]{imcond}
Martin, R. and Liu, C. (2015a).
\newblock Conditional inferential models: combining information for prior-free
  probabilistic inference.
\newblock {\em J. R. Stat. Soc. Ser. B. Stat. Methodol.}, 77(1):195--217.

\bibitem[Martin and Liu, 2015b]{immarg}
Martin, R. and Liu, C. (2015b).
\newblock Marginal inferential models: prior-free probabilistic inference on
  interest parameters.
\newblock {\em J. Amer. Statist. Assoc.}, 110(512):1621--1631.

\bibitem[Martin and Liu, 2016]{imbook}
Martin, R. and Liu, C. (2016).
\newblock {\em Inferential models}, volume 147 of {\em Monographs on Statistics
  and Applied Probability}.
\newblock CRC Press, Boca Raton, FL.
\newblock Reasoning with uncertainty.

\bibitem[Meng, 2017]{xl.bff.2017}
Meng, X.-L. (2017).
\newblock {X-L Files: Bayesian, Fiducial and Frequentist: BFF4EVER}.
\newblock {\em IMS Bulletin}, 46(4):17.

\bibitem[Mickey and Brown, 1966]{mickey.brown.1966}
Mickey, M.~R. and Brown, M.~B. (1966).
\newblock Bounds on the distribution functions of the {B}ehrens-{F}isher
  statistic.
\newblock {\em Ann. Math. Statist.}, 37:639--642.

\bibitem[Molchanov, 2005]{molchanov2005}
Molchanov, I. (2005).
\newblock {\em Theory of Random Sets}.
\newblock Probability and Its Applications (New York). Springer-Verlag London
  Ltd., London.

\bibitem[Nadarajah et~al., 2015]{nadarajah.etal.2015}
Nadarajah, S., Bityukov, S., and Krasnikov, N. (2015).
\newblock Confidence distributions: a review.
\newblock {\em Stat. Methodol.}, 22:23--46.

\bibitem[Neyman, 1941]{neyman1941}
Neyman, J. (1941).
\newblock Fiducial argument and the theory of confidence intervals.
\newblock {\em Biometrika}, 32:128--150.

\bibitem[Nguyen, 2006]{nguyen.book}
Nguyen, H.~T. (2006).
\newblock {\em An Introduction to Random Sets}.
\newblock Chapman \& Hall/CRC, Boca Raton, FL.

\bibitem[Reid and Cox, 2015]{reid.cox.2014}
Reid, N. and Cox, D.~R. (2015).
\newblock On some principles of statistical inference.
\newblock {\em Int. Stat. Rev.}, 83(2):293--308.

\bibitem[Scheff{\'e}, 1970]{scheffe1970}
Scheff{\'e}, H. (1970).
\newblock Practical solutions of the {B}ehrens--{F}isher problem.
\newblock {\em J. Amer. Statist. Assoc.}, 65:1501--1508.

\bibitem[Schweder and Hjort, 2002]{schweder.hjort.2002}
Schweder, T. and Hjort, N.~L. (2002).
\newblock Confidence and likelihood.
\newblock {\em Scand. J. Statist.}, 29(2):309--332.

\bibitem[Schweder and Hjort, 2013]{schweder.hjort.2013}
Schweder, T. and Hjort, N.~L. (2013).
\newblock Discussion: ``{C}onfidence distribution, the frequentist distribution
  estimator of a parameter: a review'' [mr3047496].
\newblock {\em Int. Stat. Rev.}, 81(1):56--68.

\bibitem[Schweder and Hjort, 2016]{schweder.hjort.book}
Schweder, T. and Hjort, N.~L. (2016).
\newblock {\em Confidence, Likelihood, Probability}, volume~41 of {\em
  Cambridge Series in Statistical and Probabilistic Mathematics}.
\newblock Cambridge University Press, New York.

\bibitem[Seidenfeld, 1992]{seidenfeld1992}
Seidenfeld, T. (1992).
\newblock R. {A}. {F}isher's fiducial argument and {B}ayes' theorem.
\newblock {\em Statist. Sci.}, 7(3):358--368.

\bibitem[Shafer, 1976]{shafer1976}
Shafer, G. (1976).
\newblock {\em A Mathematical Theory of Evidence}.
\newblock Princeton University Press, Princeton, N.J.

\bibitem[Shafer, 1979]{shafer1979}
Shafer, G. (1979).
\newblock Allocations of probability.
\newblock {\em Ann. Probab.}, 7(5):827--839.

\bibitem[Shafer, 1987]{shafer1987}
Shafer, G. (1987).
\newblock Belief functions and possibility measures.
\newblock In Bezdek, J.~C., editor, {\em The Analysis of Fuzzy Information,
  Vol. 1: Mathematics and Logic}, pages 51--84. CRC.

\bibitem[Spj{\o}tvoll, 1983]{spjotvoll1983}
Spj{\o}tvoll, E. (1983).
\newblock Preference functions.
\newblock In {\em A {F}estschrift for {E}rich {L}. {L}ehmann}, Wadsworth
  Statist./Probab. Ser., pages 409--432. Wadsworth, Belmont, Calif.

\bibitem[Stein, 1959]{stein1959}
Stein, C. (1959).
\newblock An example of wide discrepancy between fiducial and confidence
  intervals.
\newblock {\em Ann. Math. Statist.}, 30:877--880.

\bibitem[Trafimowa and Marks, 2015]{pvalue.ban}
Trafimowa, D. and Marks, M. (2015).
\newblock Editorial.
\newblock {\em Basic Appl. Soc. Psych.}, 37(1):1--2.

\bibitem[Wasserman, 2006]{wasserman2006book}
Wasserman, L. (2006).
\newblock {\em All of Nonparametric Statistics}.
\newblock Springer Texts in Statistics. Springer, New York.

\bibitem[Wasserstein and Lazar, 2016]{wasserstein.lazar.asa}
Wasserstein, R.~L. and Lazar, N.~A. (2016).
\newblock The {ASA}'s statement on {$p$}-values: context, process, and purpose
  [{E}ditorial].
\newblock {\em Amer. Statist.}, 70(2):129--133.

\bibitem[Xie and Singh, 2013]{xie.singh.2012}
Xie, M.-g. and Singh, K. (2013).
\newblock Confidence distribution, the frequentist distribution estimator of a
  parameter: a review.
\newblock {\em Int. Stat. Rev.}, 81(1):3--39.

\bibitem[Zabell, 1992]{zabell1992}
Zabell, S.~L. (1992).
\newblock R. {A}. {F}isher and the fiducial argument.
\newblock {\em Statist. Sci.}, 7(3):369--387.

\bibitem[Zadeh, 1978]{zadeh1978}
Zadeh, L.~A. (1978).
\newblock Fuzzy sets as a basis for a theory of possibility.
\newblock {\em Fuzzy Sets and Systems}, 1(1):3--28.

\end{thebibliography}

\end{document}